\documentclass[11pt,a4paper]{article}

\usepackage[latin1]{inputenc}

\usepackage[normalem]{ulem}

\usepackage{amssymb}

\usepackage{graphicx}

\usepackage{graphicx,indentfirst,amsmath,amsfonts,amssymb,amsthm,newlfont}
\usepackage{epsfig}

\newtheorem{theorem}{Theorem}[section]
\theoremstyle{plain}
\newtheorem{proposition}[theorem]{Proposition}

\newtheorem{lemma}[theorem]{Lemma}
\newtheorem{corollary}[theorem]{Corollary}

\theoremstyle{definition}
\newtheorem{example}[theorem]{Example}
\newtheorem{definition}[theorem]{Definition}

\theoremstyle{remark}

\newcommand{\dif}{\mathrm{Diff}}
\newcommand{\difH}{\mathrm{Diff}(\Delta^H, M)}
\newcommand{\difV}{\mathrm{Diff}(\Delta^V, M)}
\newcommand{\dete}[2]{ {  \begin{array}{c}{ {}_{#1}}\\
\mathrm{det}\\ { {}^{#2}}\end{array}}}
\newcommand{\Prob}{\mathbf P}
\newcommand{\R}{\mathbf R}
\newcommand{\N}{\mathbf N}
\newcommand{\Z}{\mathbf Z}

\newcommand{\F}{\mathcal F}

\begin{document}
\begin{center}


 {\Large {\bf Topology of foliations and decomposition of
stochastic flows of diffeomorphisms\\[3mm]
 }}

\end{center}

\vspace{0.3cm}

\begin{center}
{\large Alison M. Melo\footnote{E-mail: alison.melo@univasf.edu.br.
Research
supported by CNPq 142084/2012-3.} \ \ \ \ \ \ \ \ \ \ \ \ \
 { Leandro Morgado}\footnote{E-mail:
leandro.morgado@ufsc.br.
Research supported by  FAPESP 11/14797-2.}

\medskip

{ Paulo R. Ruffino}\footnote{Corresponding author, e-mail:
ruffino@ime.unicamp.br.
Partially supported by FAPESP  12/18780-0,
11/50151-0 and CNPq 477861/2013-0.}}

\vspace{0.2cm}

\textit{Departamento de Matem\'{a}tica, Universidade Estadual de Campinas, \\
13.083-859- Campinas - SP, Brazil.}

\end{center}

\begin{abstract}
Let $M$ be a compact manifold equipped with a pair of complementary
foliations, say horizontal and vertical. In Catuogno, Silva and 
Ruffino \cite{CSR2} it is shown that, up
to a stopping time $\tau$, a stochastic flow of local diffeomorphisms
$\varphi_t$ in $M$ can be written as  a 
Markovian process in the subgroup of diffeomorphisms which preserve the 
horizontal foliation composed with a process in the subgroup of 
diffeomorphisms 
which preserve the vertical foliation. Here, we discuss
topological 
aspects of this decomposition. The main result guarantees the 
global 
decomposition of a flow if it preserves the orientation of a
transversely orientable  foliation. In the  last 
section, we present an It\^o-Liouville formula for subdeterminants of 
linearised 
flows. We use this formula to obtain sufficient conditions for 
the existence of the 
decomposition for all $t\geq 0$.
\end{abstract}


\vspace{0.3cm}
\noindent {\bf MSC2010 subject classification:} 60H10, 58J65, 57R30.

\section{Introduction}

Consider a stochastic flow $\varphi_t$ of diffeomorphisms in a compact 
differentiable
manifold $M$ endowed with some structure (Riemannian, Hamiltonian,
foliation, etc). In many situations, the decomposition of $\varphi_t$ 
with components in subgroups of
the group of diffeomorphisms $\dif(M)$ provide interesting dynamical or 
geometrical information of the stochastic system. In the literature, this kind 
of decomposition has been studied in several  frameworks and with different
aimed subgroups; among others, see e.g. Bismut \cite{Bismut}, Kunita 
\cite{Kunita-1}, \cite{Kunita-2},  Ming Liao \cite{ML} and some of our previous 
works \cite{CSR}, \cite{Colonius-Ruffino}, \cite{MR}, \cite{Ruffino}. 

In particular, in Catuogno, da Silva and Ruffino \cite{CSR2}, the authors
consider a pair of complementary distributions in a differentiable manifold 
$M$, in the sense that each tangent space splits into a direct sum of two 
subspaces depending differentiably on $M$. These subspaces are called, by
convenience, horizontal and vertical distributions. In \cite{CSR2} it is shown
that locally, up to a stopping time $\tau$, a stochastic flow $\varphi_t$ in
$M$ can be decomposed as $\varphi_t = \xi_t \circ \psi_t$, where $\xi_t$ is a
diffusion in the group of diffeomorphisms $\difH$ generated by horizontal
vector fields, and $\Psi_t$ is a process in the group of diffeomorphisms 
$\difV$ 
generated by vertical vector fields. The authors also present stochastic 
differential equations on the corresponding infinite dimensional Lie
subgroups for the components $\xi_t$ and $\psi_t$. The infinite dimensional Lie 
group structure considered in this case is described in Milnor \cite{Milnor}, 
Neeb \cite{Neeb} and Omori \cite{Omori}. The stopping time $\tau$ mentioned above, which restricts the
time where the decomposition exists, appears due to an explosion in the equation 
of one of the components of the decomposition, with initial conditions at the 
identity. It is  related to the degeneracy of the dynamics of one distribution with respect to the other.

The initial motivation for this kind of decomposition comes from a system whose flow is originally energy perserving, hence trajectories lies on energy levels. After a perturbation by transverse vector fields, hence destroying this foliated behaviour, the decomposition allows to study separatly an energy preserving component and a transverse component. This is, for instance the context of an averaging principle in Gargate-Ruffino \cite{Gargate-Ruffino}, where the vertical component is rescaled by $\epsilon^{-1}$, see also Li \cite{Li} in the Hamiltonian context. Flows in a principal fibre bundle with an affine connection gives another class of examples where the distributions are not necessarily integrable, hence generating holonomy.  In Melo, Morgado and Ruffino 
\cite{MMR}, the same decomposition is considered for the case of stochastic 
flows with jumps, using Marcus equation, as in Kurtz, Pardoux and Protter 
\cite{KPP}.

In this article, we work with the same structure of \cite{CSR2}, but
assuming that the distributions are integrable:  The manifold $M$ is endowed with
two complementary foliations $\mathcal{H}$ and  $\mathcal{V}$, i.e., such that 
the leaves of the vertical foliation  $\mathcal{V}$ are
transverse
to the leaves of the horizontal foliation $\mathcal{H}$, in the sense that
$TM= T\mathcal{H}\oplus T\mathcal{V}$. We use the notation $(M,\mathcal{H},
\mathcal{V})$ for this space. The action of the subgroup of diffeomorphisms 
$\difH $ 
fixes each horizontal leaf and the action of $\difV $ fixes  vertical 
leaves.  The dynamics in $M$ is given by a stochastic flow of (local) diffeomorphisms
$\varphi_t$ generated by a Stratonovich SDE on $M$:

\begin{equation} dx_{t} \ =  \sum_{r=0}^{m}
X_{r}(x_{t})\circ dW_{t}^{r} \label{equacao original},
\end{equation}
where $W^0_t = t$, $ (W^1, \ldots , W^m)$ is a Brownian motion in $\R^m$
constructed on a filtered complete probability space
$(\Omega, \F, \F_t, \Prob)$ and $X_0, X_1, \ldots , X_m$ are smooth vector
fields in $M$. In this situation, there exists a stochastic solution flow
of (local)
diffeomorphisms $\varphi_t$, see e.g. among others the classical Kunita
\cite{Kunita-1}, Elworthy \cite{elworthy82}. The flow is assumed to be
complete, such that its explosion time is not a restriction for the 
decomposition. The fact that we deal with an stochastic flow of a 
Stratonovich SDE 
is convenient, say, to get explicit equations for the components $\xi_t$ and 
$\psi_t$, and to find an It\^o formula for subdeterminants involved in the 
context (Section 3.1). Nevertheless, here, many of our results on decomposition 
apply also to a continuous family $\phi_t$ of diffeomorphisms such that 
$\phi_0=Id$, 
which not necessarily satisfies the cocycle property (e.g. Theorem \ref{Thm: 
globdec}).

Our aim here is to study topological features on the foliations 
relating to this decomposition. It is particularly interesting the fact that 
decomposability
of a flow is strongly related with the geometrical concept of transverse
orientation in a pair of foliations. Precisely, the influence of the 
topology of the foliations appears  in two intertwined  categories: analytical 
and topological aspects. The analytical approach is presented in Section 1.2. 
It 
is essentially guided by  the subdeterminant of the linearised flow: 
 it gives us a necessary and sufficient condition for the local existence of 
the decomposition.

For the topological aspects, initially, we consider a pair of horizontal 
and vertical 
foliations, where  there might be a set of points which one can 
not 
reach from a point $x_0$ in the manifold $M$, by taking a concatenation of a 
vertical path with a horizontal path, in this order. We discuss this property 
of 
attainability and its consequences for the decomposition in Section 2.1. In 
Section 2.2, we study the dynamics of the stochastic flow $\varphi_t$ 
in the leaves, and how this action might be a restriction for the 
decomposition. 
Our main result is in Section 2.3, where we proof that, under the 
condition of transverse orientability of the horizontal foliation, the global 
decomposition exists if and only if the family of diffeomorphisms preserves this orientation.

In the last section we present 
an It\^o-Liouville formula for subdeterminants of the  linearized stochastic flow $D\varphi_t$. 
Using this formula and  Cauchy-Binet identity (see e.g. 
Tracy and Widom \cite{Tracy}), we discuss a pair of sufficient condition for the existence 
of the decomposition of the flow $\varphi_t$, for $t\geq 0$.

\subsection{Foliations}
We recall briefly some geometrical facts about
foliations in a differentiable manifold. For more details, see e.g., among
many others, Candel and Conlon \cite{Candel-Conlon}, Tamura \cite{Tamura},
Walcak \cite{Walcak}. Let $M$ be a Riemannian manifold of dimension
$n$. 

\begin{definition}  A {\it foliation} with codimension $k$ in $M$
is a partition $\mathcal{F}=\{\mathcal{F}(p):p\in M\}$
of $M$ endowed with
a $C^{\infty}$-atlas $\{(\phi_{\alpha}, U_{\alpha})\}_{\alpha\in \mathcal I}$
where
$$
\begin{array}{crcl}\phi_{\alpha}:&U_{\alpha}&\rightarrow
&\mathbb{R}^{n-k}\times
\mathbb{R}^{k}\\
                                               & p&\mapsto & (x_{\alpha}(p),
y_{\alpha}(p)),
\end{array}
$$
 such that each local chart $\phi_{\alpha}$ satisfies the property that if $
y_{\alpha}(p)=y_{\alpha}(q)$, then $ \mathcal{F}(p)=\mathcal{F}(q),
$
for $p$ and $q$ in $U_{\alpha}$.
\end{definition}

The atlas  $\{\phi_{\alpha}, U_{\alpha}\}$ above is called a {\it foliated
atlas}.
Each $\mathcal{F}(p)$
is called the {\it leaf} of $p$.
A set $P\subset M$ is a {\it plaque} of the foliation $\mathcal{F}$ if it is an
open submanifold of an $(n-k)$-dimensional leaf, precisely: $P$ has the form
$P=\phi_{\alpha}^{-1}(D)$ where $D$ is an open disk  of dimension $n-k$
contained in a level subset $\{p\in U_{\alpha}:y_{\alpha}(p)=y_0\}$.
Each leaf of the foliation is the image of an immersion of a complete
manifold into $M$. Given  $N\subset M$ the saturation $\mathcal{F}(N)$ of $N$
by $\mathcal{F}$ is the set $\mathcal{F}(N)= \cup_{p\in N}\mathcal{F}(p)$.

With refinements, if necessary, one can always obtain a regular
atlas $\{\phi_{\alpha}, U_{\alpha}\}$, i.e. such that $U_{\alpha}$ is
precompact in a bigger foliated domain, the cover $\{U_{\alpha}\}$ is locally
finite and the interior of each closed plaque of $U_{\alpha}$ meets at most one
plaque in the closure of $U_{\beta} $, see \cite[Chap.1]{Candel-Conlon}.

Given two local charts $(\phi_1,U_1)$ and $(\phi_2,U_2)$ in the  regular
foliated
atlas $\{(\phi_{\alpha},U_\alpha)\}$ such that $U_1\cap U_2\not=\emptyset$
we define the functions $y_i:\phi_2( U_1\cap U_2) \rightarrow \mathbb{R}$ by
$\phi_1\circ\phi_2^{-1}=(y_1,...,y_n).$
Let $x_i$ denote the usual coordinates on $\mathbb{R}^n$,
we say that the atlas $\{(\phi_{\alpha},U_\alpha)\}$  is {\it transversely
orientable}
if
$$ 
\det 
\displaystyle\frac{\partial(y_{n-k+1},...,y_n)}{\partial(x_{n-k+1},..., x_n)
}> 0
$$
everywhere in the domain, for any two local charts $(\phi_1,U_1)$ and
$(\phi_2,U_2)$ in
$\{(\phi_{\alpha},U_\alpha)\}$
with  $U_1\cap U_2\not=\emptyset$. See an example of non-transversely
orientable foliation in Example \ref{Ex: non-orientable} below.

An important property we use in the next section is the uniform
transverselity  of the foliation, see e.g. Camacho and Lins-Neto
\cite{Camacho-Lins}:


\begin{theorem}\label{Thm: transunif} Consider a foliated space $(M, 
\mathcal{F})$
and a fixed leaf
$F$. Given two  $k$-dimensional
 submanifolds  $N_1$ and $N_2$ which are transverse to $F$, there exist
disks $D_1\subset N_1$
and $D_2\subset N_2$ and a
diffeomorphism
$f:D_1\rightarrow D_2$ such that for any leaf $F'$ with $F'\cap D_1 \not=
\emptyset$ we have that
$f(F' \cap D_1)= F'\cap D_2 $.
 \end{theorem}

For a pair of complementary foliations in $M$ we have:

\begin{definition}  A birregular atlas on
$(M,\mathcal{H},\mathcal{V})$ is an atlas
$\{(\phi_{\alpha},y_\alpha,U_\alpha)\}_{\alpha \in \mathcal{I}}$
which is simultaneously a regular foliated atlas for $\mathcal{H}$ and
$\mathcal{V}$.
\end{definition}

It is well known that there always exists a birregular foliated atlas for
$(M,\mathcal{H},\mathcal{V})$, see e.g. \cite[Prop. 5.1.4]{Candel-Conlon}.

\subsection{Characterization of local decomposition}

Consider a diffeomorphism $\varphi:  U \rightarrow V$, with $U$ and $V$ 
open subsets of $\R^n$. The product $ \R^{n-k} \times \R^{k}$ is a canonical 
Cartesian pair of foliations of  $\R^n$. 
With respect to this product, write $\varphi = (\varphi^1 (x,y), \varphi^2 
(x,y))$, i.e.  $x$ and $ \varphi^1(x,y)$ belong to $ \R^{n-k}$ and $y, 
\varphi^2(x,y) \in \R^k$. 
An
analytical restriction for the local decomposition of
$\varphi$ appears related
with the  subdeterminant of the derivative of $\varphi$:

\begin{proposition} \label{Prop: analyt local obstruct decomp}
There exists a unique (up to reduction in the domain) decomposition $\varphi= 
\xi \circ
\psi$ in a neighbourhood of $(x,y)$ if and only if
\[
 \det
\frac{\partial \varphi^2(x,y)}{\partial y} \neq  0.
\]
Moreover, if $t\mapsto \varphi_t$ is continuous (with respect to, say, $C^k$ 
topology) and satisfies the determinant condition above, then $\xi_t$ and 
$\psi_t$ are also continuous with respect to time $t$.
\end{proposition}
\begin{proof}
 In fact, if $\varphi_t= (\varphi_t^1, \varphi_t^2)$ then 
 \[
  \psi_t = (Id_{\R^{n-k}}, \varphi_t^2),
 \]
and 
 \[
  \xi_t = ( * , Id_{\R^{k}})= \varphi_t \circ \psi_t^{-1}.
 \]
The proof is a simple application of the theorem of inverse functions, see
also Melo, Morgado and Ruffino \cite{MMR}. Continuity of  $\xi_t$ 
and $\psi_t$ follows directly from the continuity of the components in the formulae above.

\end{proof}

Applying this characterization into  flows in $\R^n$, we 
have that there exists the decomposition of an stochastic flow $\varphi_t= 
\xi_t \circ
\psi_t$  up to a stopping  time $\tau$ in a neighbourhood of the
initial condition $(x,y)$, where 
\[
 \tau = \sup \{t >0;  \det
\frac{\partial \varphi_s^2(x,y)}{\partial y}\neq 0, \mbox{ for all } 0\leq s 
\leq t \}.
\]
In particular, if  $(M, \mathcal{H, V})$ is a product space $M= H \times V$ with $H$ 
and $V$ two differentiable manifolds and  a 
diffeomorfism $\varphi$ 
sends each vertical
leaf entirely into a vertical leaf, then the determinant never vanishes,
hence the decomposition $\varphi= \xi
\circ \psi$ holds. For example, linear systems
with spherical (horizontal)  and radial (vertical)  foliations in $\R^n 
\setminus \{0\}$: in this case, vertical leaves are sent to vertical leaves, 
hence
the decomposition holds for all $t\geq 0$, see \cite{MMR}. 
On the other hand, consider the simple example of a 
linear rotation in $\R^2$ endowed with the
canonical Cartesian (horizontal and vertical) foliations. We have the 
decomposition:
\[
\varphi_t = \left(
                                \begin{array}{cc}
                                  \cos t & -\sin t \\
                                  \sin t & \cos t
                                \end{array}
                              \right)=
                              \left(
                                \begin{array}{cc}
                                  \sec t & -\tan t \\
                                  0 & 1
                                \end{array}
                              \right)
                              \left(
                                \begin{array}{cc}
                                  1 & 0 \\
                                  \sin t & \cos t
                                \end{array}
                              \right).
\]
Clearly, the vertical leaves are not preserved. There exists an 
analytical obstruction when $t=\frac{\pi}{2}$. So,
the stopping time $\tau = \frac{\pi}{2}$. In the next section, the same
example is interpreted as a topological obstruction at $t=\pi/2$, since
the dynamics of vertical leaves collapses on horizontal leaves (Proposition 
\ref{Prop: intersecao}).

 \section{Topological aspects on global decomposition }

 We distinguish three kinds of topological aspects of the foliations related 
with the existence of the decomposition: the first one concerns the limitation 
on the attainability of trajectories starting at an initial condition and 
running exclusively along a vertical trajectory concatenated with a horizontal 
trajectory. The second aspect concerns the effect of the 
dynamics on the leaves. Finally, the third aspect concerns  
transversely orientation of the horizontal foliation. Our main 
result guarantees that if the flow preserves this orientation, then 
the flow is globally decomposable.

 \subsection{Attainability}
In many pairs of foliations, given
a starting point $x_0$ there might be a set
of points which one can not reach from $x_0$ by taking a concatenation of a
vertical path with a horizontal path, in this order. See e.g. Figure 1, where,
say the horizontal leaves are represented by bold curves and vertical leaves
are represented by thin curves in $\R^2\setminus \{0\}$. For $x_0=(1,1)$, the 
attainable points is the
open set above the line $y=-x$.
\begin{figure}[htb!]\label{Teia}
\centering
\includegraphics[scale=0.16]{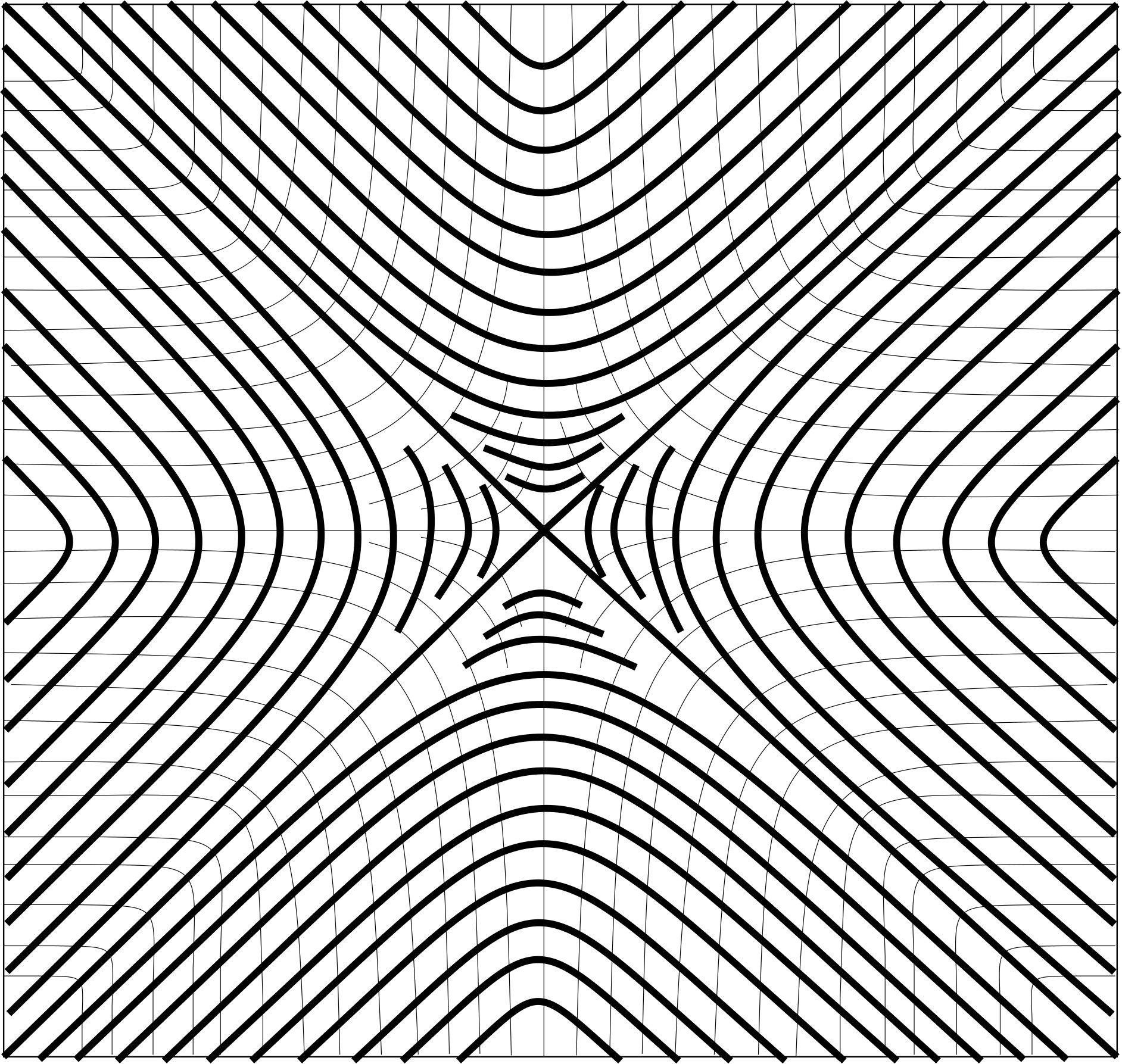}
\caption{Pair of foliations with $\mathcal{A}(x)\neq M$.}
\end{figure}

\noindent This idea leads to the following:
\begin{definition} The attainable points from $x \in M$ with respect to the
pair of foliations $(M,\mathcal{H}, \mathcal{V})$ is the set
\[
 \mathcal{A}(x)= \mathcal{H}(\mathcal{V}(x)).
\]
\end{definition}
Clearly, $\mathcal{A}(x)$ is horizontally saturated and,  if a
diffeomorphism $\varphi$ is
decomposable in a neighbourhood of $x$,
then $\varphi(x) \in
\mathcal{A}(x)$. Hence, non-attainability is an intrinsic obstruction
for the decomposition. Thinking on reversibility and  commutativity of the
decomposition, we present the following:


\begin{definition} The co-attainable set of $x\in M$ with respect to the
pair of foliations $(M,\mathcal{H}, \mathcal{V})$ is the set
$$
\mathcal{C}(x)=\mathcal{H}(\mathcal{V}(x))\cap\mathcal{V}(\mathcal{H}(x)).
$$
\end{definition}

A point $y\in M$ belongs to $\mathcal{C}(x)$ if $y \in \mathcal{A}(x)$ and $x
\in \mathcal{A}(y)$.
Note that, for
each $x \in M$,
the sets $\mathcal{A}(x)$ and $\mathcal{C}(x)$ are open since the leaves of
$\mathcal{H}$ are everywhere
 transverse to the leaves of $\mathcal{V}$.
Next we present a property of the attainable sets  $\mathcal{A}(x)$.

\begin{proposition} Let $M$ be a compact connected manifold.
If for a certain $x\in M$ we have $\mathcal{A}(x)=\mathcal{C}(x)$, then
$\mathcal{A}(x)=M$.
\end{proposition}

\begin{proof}
Consider $y\in \partial\mathcal{A}(x)$.
Then $\mathcal{V}(y)\cap \mathcal{A}(x)\not=\emptyset,$
since $\partial \mathcal{A}(x)$ is an $\mathcal{H}$-saturated set.
Therefore, there exists a  $z\in \mathcal{V}(y)\cap \mathcal{A}(x)$.
On the other hand, as  $z\in \mathcal{C}(x)$,
we have that $\mathcal{V}(y)\cap \mathcal{H}(x)\not=\emptyset$.
Now let $\gamma$ be a horizontal curve starting at $x$ and ending at
the point $w\in\mathcal{V}(y)\cap \mathcal{H}(x)$.

Using the fact that $\partial\mathcal{A}(x)$ is $\mathcal{H}$-saturated
and that $M$ is compact we have
 that either all vertical leaves of $\mathcal{V}(\gamma)$
intersect $\partial\mathcal{A}(x)$ or none of them intersects
$\partial\mathcal{A}(x)$.
But we already know that $\mathcal{V}(w)=\mathcal{V}(y)$ intersects
$\partial\mathcal{A}(x)$, therefore so does $\mathcal{V}(x)$.
This implies that $\partial\mathcal{A}(x)\subset \mathcal{A}(x)$ and then
$\mathcal{A}(x)=M$, since  $M$ is connected.

\end{proof}

The converse of the proposition above obviously  does not hold. Also, if
$\mathcal{A}(x)=M$ it does not imply that $\mathcal{A}(z)=M$ for all $z\in M$.
As a simple example, in Figure 2, $\mathcal{A}(x)=M$ but
$y\notin \mathcal{A}(z)$.
\begin{figure}[htb!]\label{L}
\centering
\includegraphics[scale=0.2]{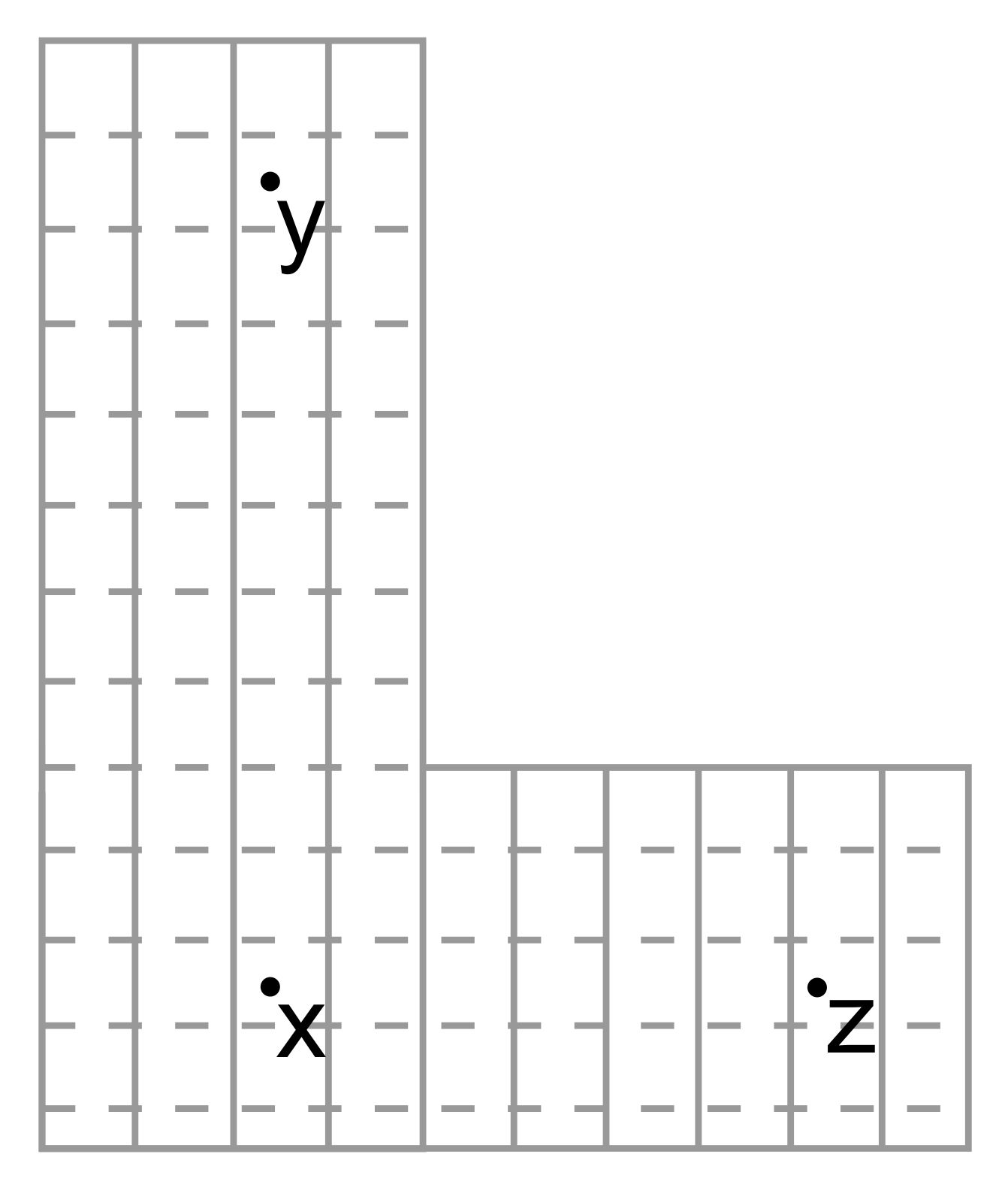}
\caption{ $\mathcal{A}(x) \neq \mathcal{A}(z)$.}
\end{figure}

Naturally, 
if a diffeormorfism $\phi$ is such 
that for some $x\in M$, we have that  $\phi(x)\not\in \mathcal{A}(x)$ then 
$\phi$ 
is not decomposable in a neighbourhood of $x$.

 \subsection{Dynamical  obstruction}

Next Proposition contains properties of the dynamics of the leaves along the
components of a decomposable flow. These simple properties explain
topologically the nondecomposability of some diffeomorphisms.

\begin{proposition}  \label{Prop: intersecao}
Suppose that $\varphi_t$ has the
decomposition $\varphi_t=\xi_t \circ \psi_t$ up to a stopping time $\tau>0$
a.s.. Then, for $0\leq t<\tau$,  $\psi_t(x)\in
\mathcal{V}(x) \cap \mathcal{H}(\varphi_t(x)) $ and $\xi_t(y)\in
\mathcal{H}(y) \cap \varphi_t(\mathcal{V}(y)) $ where $x$ and $y$ are in
the appropriate domain.

\end{proposition}
\begin{proof}
 The first statement follows easily from the fact that $\psi_t(x) \in
\mathcal{V}(x)$ and $\xi_t$ preserves the horizontal leaves.
 The second statement follows obviously from the fact that
$\xi_t(y)\in \mathcal{H}(y)$ and $y\in \mathcal{V}(y)$.

\end{proof}

Consider the plane $\R^2$ endowed with the usual vertical and horizontal foliations.
A rotation by $\pi/2$  does not satisfy the second property of Proposition \ref{Prop: 
intersecao}.  Hence dynamical obstruction is independent of attainability
since $\mathcal{A}(p)= \R^2$ for all $p\in \R^2$.

\subsection{Preserving transverse orientation }

Suppose that the horizontal foliation is transversely orientable and fix a 
biregular atlas with positive transverse orientation for the horizontal
foliation. Given  a continuous family $\phi_t$ of (global) diffeomorphisms on 
$M$ with
$\phi_0=1d$ and a point  $p\in M$, consider two  local
coordinate systems in this atlas:
 $\eta_p: U_p\subset M \rightarrow  \R^{n-k} \times \R^{k}$, in a
neighbourhood of $p$ and  $\eta_{\phi_t(p)}: U_{\phi_t(p)} \subset M
\rightarrow  \R^{n-k} \times \R^{k}$,
in a neighbourhood of $\phi_t(p)$.
With respect to these coordinate systems one writes the family of
diffeomorphisms locally as $\phi_t = (\phi_t^1 (x,y), \phi_t^2
(x,y))$, such that
\begin{equation} \label{Eq: subdeterminant}
\dete{n-k+1,\ldots,n}{n-k+1,...,n} D\phi_t(x) =   \det
\frac{\partial \phi_{t}^2(x,y)}{\partial y},
\end{equation}
where the determinant above is taken for the  $k\times k$ submatrices defined 
by 
the intersections of the columns indicated by subscripts and rows indicated 
by 
the superscripts.
Distinct choices of coordinate systems in this atlas does not change the sign
of the determinant above. Globally we have the following results whose main
application is when $\phi_t$ is a flow of diffeomorphisms.

\begin{theorem}\label{Thm: globdec} Suppose that the foliations $(M,
\mathcal{H}, \mathcal{V})$ is transversely
orientable for the horizontal
foliation.
Let $(\phi_t)_{t\in [0,a)}$ be a family
of
diffeomorphisms on $M$ depending continuously on $t$, with $\phi_0=1d$.
The family of diffeomorphisms $\phi_t$ is globally decomposable for all  $t\in 
[0,a)$ if and only if it preserves the transverse orientation,
i.e.
\[
\dete{n-k+1,...,n}{n-k+1,...,n} D\phi_t(x)> 0,
\]
for any $x\in M$ and for all $t\in [0,a)$.
\end{theorem}

\begin{proof} The idea of the proof is to show that if the global decomposition 
exists up to a time $t_0 <a$, then the decomposition can be extended to $t_0 
+ \delta$, with a positive $\delta$. This construction involves changing of 
local charts which are horizontal and vertically compatible with each other. 
These coherent local charts are paving the next $\delta$-time 
future trajectories as the system evolves.

For each $p\in M$, take a bounded open set $U_p$ which 
belongs to
the
biregular atlas. Define  $t_0=\sup\{t\in [0,a):
\phi_t|_{U_p} \mbox{ is decomposable as } \phi_t=\xi_t \circ \psi_t \}$. 
Suppose that $t_0<a$. Initially we prove that $\phi_{t_0}$ is 
locally decomposable. Since $p$ and $\phi_{t_0}(p)$ might be in distinct local 
chart, Proposition \ref{Prop: analyt local obstruct decomp} does not apply 
directly.

For each $t < t_0$, we claim that there exist three coordinate
systems $((x_{\alpha},y_{\alpha}),U_{\alpha})$, $\alpha=1,2$ or 3, covering 
$p$, $\psi_t(p)$ and
$\phi_t (x)$, respectively, defined on open sets $U_{p}, U_{\psi(p)}$ and 
$U_{\phi(p)}$ which are {\it coherent} in $\R^n$ in the sense that, distinct
points in the same horizontal leaf  have the same $y_{\alpha}$ coordinate;
analogously to the vertical leaves. Precisely (omitting the subscript $t$):


\begin{itemize}
\item[i)] If $u\in U_{p}$ and $v\in U_{\psi(p)}$ with
$\mathcal{V}(u)=\mathcal{V}(v)$ then $x_p(u)=x_{\psi(p)}(v);$
\item[ii)]  If $u\in U_{\psi(p)}$ e $v\in U_{\phi(p)}$ with
$\mathcal{H}(u)=\mathcal{H}(v)$ then $y_{\psi(p)}(u)=y_{\phi(p)}(v)$.
\end{itemize}
In fact, using that there exists a biregular coordinate system
$\eta_{\psi(p)}: U_{\psi(p)} \rightarrow \R^n$ in a neighbourhood of $\psi(p)$,
consider  $ \eta_{p}=\eta_{\psi(p)}\circ
\psi_{p}$ and $ \eta_{\phi(p)}=\eta_{\psi(p)}\circ\xi^{-1}$. With this choice of
local coordinates, properties  (i) and (ii) above follow directly from the fact
that $\psi$ preserves leaves of $\mathcal{V}$ and $\xi$ preserves  leaves of
$\mathcal{H}$.

 By continuity of $\phi_t$ we have that
for each $q\in U_p$ there exists an $\epsilon>0$, sufficiently 
small,
such that $\phi_{t_0}(q)$ is attainable from $U_p$ at time 
$(t_0-\epsilon)$, i.e.
$\phi_{t_0}(q)\in \phi_{t_0- \epsilon}(U_p)$.
Consider the local diffeomorphism on $\R^n$ given by
$$
\tilde{\phi}_q := \eta_{\phi(p)}\circ \phi_{t_0}\circ\eta^{-1}_p,
$$
where the coherent coordinate systems $\eta_{\phi(p)}$ and  $\eta_p$ above
refer to time $(t_0-\epsilon)$. By hypothesis
$$
\dete{n-k+1,...,n}{n-k+1,...,n}(D\tilde{\phi}_q)>0.
$$
Therefore, by Proposition \ref{Prop: analyt local obstruct decomp},
$\tilde{\phi}_q$ has a unique decomposition
$\tilde{\phi}_q=\tilde{\xi}_q\circ \tilde{\psi}_q$ in $\R^{n-k} \times \R^k$
such that the component
$\tilde{\xi}_q$ preserves vertical coordinates and $\tilde{\psi}_q$
preserves horizontal coordinates. Take
$$
\psi_q=\eta_{\psi(p)}^{-1}\circ\tilde{\psi}_q\circ
\eta_{p} \ \ \ \mbox{ and }\ \ \
\xi_q=\eta_{\phi(p)}^{-1}\circ\tilde{\xi}_q\circ \eta_{\psi(p)}.
$$
The uniqueness of the local decomposition and the coherency of the
coordinate systems allow us to glue all the
$\psi_q$, $q\in U_p$,
together into a diffeomorphism $\psi_p$,
defined all over $U_p$. Analogously, we construct the
horizontal
component $\xi_p$ defined all over $U_{\psi(p)}$. Then $\phi_{t_0}$ is 
decomposable
in the neighbourhood $U_p$ of $p$, for all $p\in M$.

Now, suppose by contradiction that $t_0< a$. We are going to
show that there exist a sufficiently small
$\delta>0$ such that $\phi_{t_0+\delta}$ is also decomposable. This 
finishes the proof, since the converse is trivial. Since
$U_{\psi(p)}$  is birregular (here the subscript
$t_0$ is omitted), it can be restricted such that its image by $\eta_{\psi(p)}$
is an $n$-dimensional rectangle whose  border $\partial{U_{\psi(p)}}$ of
$U_{\psi(p)}$
 is the union of purely horizontal components $\partial^{h}U_{\psi(p)}$
and  purely vertical components $\partial^{v}U_{\psi(p)}$.
In this case, the trajectory of $\psi_t(q)$, along the time $t$,  leaves
$U_p$,
gets in and out of $U_{\psi(p)}$ only by crossing their horizontal borders.

By Theorem \ref{Thm: transunif}, the set $U_{\psi(p)}$ can be extended 
vertically to
$U'_{\psi(p)}$ (extending vertically in the neighbourhood of each point of
the compact $ \partial^h U_{\psi(p)}$). For the set $U_{\phi(p)}$, each
point in the border $\partial U_{\phi(p)}$ has either a horizontal or a
vertical leaf crossing the border $\partial U_{\phi(p)}$ at this point. Then,
again, by Theorem \ref{Thm: transunif}, this set can be enlarged to an open set
$U'_{\phi(p)}$ which contains the original one in its interior. Hence, the 
coordinate charts in the sets
$U'_{\psi(p)}$ and $U'_{\phi(p)}$ are still coherent.
So that, for $\delta>0$ sufficiently small, $\phi_{t_0+\delta}(U_p)\subset
U'_{\phi(p)}$. This allows the decomposition to be performed at $t_0+ 
\delta$. We conclude that $t_0=a$.

Finally, the global result is obtained using that at any point $p\in M$, the 
decomposition of 
$\phi_t|_{U_p}$, in a neighbourhood $U_p$ of $p$ holds for all $t\in [0,a)$. 
Using the uniqueness of local
decomposition (Proposition \ref{Prop: analyt local obstruct decomp}) in the 
intersections of the neighbourhoods $U_p$, for all $p\in M$, we obtain the 
global decomposition for all $t\in [0,a)$.

 \end{proof}
Transverse orientation in the proof above is crucial to 
guarantee that the sign of
$\det_{n-k+1,...,n}^{n-k+1,...,n}D \phi_t$ is  globally defined.

\begin{corollary} Suppose that the foliations $(M,
\mathcal{H}, \mathcal{V})$ is transversely
orientable for the horizontal
foliation.  Given $x \in M$, if $\varphi_t(x)$ approaches the boundary $\partial
\mathcal{A}(x)$ of the attainable set, then the subdeterminant goes to zero.
\end{corollary}

\begin{proof} Suppose that there exists $x\in M$ and $t_0>0$ such
that $\varphi_{t_0}(x)\in \partial \mathcal{A}(x)$.  Then $\phi_{t_0}$ is not
decomposable. Therefore
$\mbox{det}^{n-k+1,...,n}_{n-k+1,...,n}(D\phi_{t_0}(x))=0$ by Theorem
\ref{Thm: globdec}.

\end{proof}
The foliations in figure 1 illustrates the result stated in this corollary. The 
horizontal foliation is transversely orientable. Hence, any flow which carries 
an $x_0$ to the boundary $\partial \mathcal{A}(x_0)$ has the subdeterminant 
going to zero as it approaches $\partial \mathcal{A}(x_0)$.

\begin{example} \label{Ex: non-orientable}  Theorem above states 
conditions for  
decomposability in transversely orientable horizontal foliation.  
Consider the quotient manifold $M=[0,1]^3/\hspace{-2mm}\sim $ where the 
projection is 
taken under the identification of two faces of the cube:
\end{example}
\[
 (x,0,z) \sim (1-x, 1, 1-z)
\]
such that the leave $(x,y,1/2)$ turns into a M\"obius 
strip $S$. Hence $M$ is a tubular neighbourhood of this M\"obius strip with 
horizontal foliation $\mathcal{H}$ given by the image of the horizontal plaques 
and vertical foliation $\mathcal{V}$ is the corresponding image of the vertical 
lines.
Although $(M, \mathcal{H})$ is not transversely orientable, the 
connected foliated space $(M\setminus S, \mathcal{H})$ is 
transversely orientable. 

Consider the complete flow given by the projection of $\varphi_t (x,y,z)= 
(x,y+t,z)$. With respect to this pair of foliation, $\varphi_t$ is 
a horizontal flow, hence we have a trivial decomposition 
$\varphi_t=\xi_t \circ \psi_t$ given by $\xi_t=\varphi_t$ and $\psi_t \equiv 
Id$ for small $t$. When we consider the non-transversely orientable 
foliation $(M, \mathcal{H})$, for a local biregular chart in a neighbourhood 
of an initial condition 
$x_0 \in S$, just before times $t \in \{(2k+1) 2\pi, \mbox{ with } k\in \Z\}$ 
the decomposition does not exists as continuous trajectories in the subgroups 
of diffeomorphisms: 
since $\varphi_t$ reverses simultaneously the orientation of both the vertical 
and horizontal leaves passing through $x_0$,  before any of these times, the 
decomposition  $\varphi_t= \xi_t 
\circ \psi_t$ breaks continuity: it is given by the projection of the 
two reverting orientation diffeomorphisms $\xi_t(x,y,z)= 
(y,x,z)$ and $\psi_t(x,y,z)= (x,y,1-z)$. However, in the manifold $(M\setminus 
S, 
\mathcal{H})$ the decomposition is guaranteed by Theorem \ref{Thm: globdec} for 
all $t\geq 0$.


\section{Linear algebra of subdeterminants}

In this section, we
obtain an It\^o formula for the subdeterminants in expression (\ref{Eq:
subdeterminant}) for  the linearised solutions of Stratonovich SDE. With this 
formula we obtain sufficient conditions for the
decomposition of stochastic flows for all $t\geq 0$. Initially, we present some basic results on
linear algebra, including the well-known Cauchy-Binet formula for the
determinant of a product of matrices.

\begin{lemma} Let $A, B: \R^{n-k} \oplus \R^k  \rightarrow \R^{n-k} \oplus \R^k 
$ be
linear
isomorphisms which are block diagonals preserving the subspaces $\R^{n-k}\times \{0\}$ and 
$\{0\}
\times \R^k$. Then, for any linear operator $Y$ in $\R^{n}$, the
subdeterminant
\[
 \dete{n-k+1,...,n}{n-k+1,...,n}(A\cdot Y \cdot B^{-1})  = 
 \left[\dete{n-k+1,...,n}{n-k+1,...,n}
(Y) \right] \left[\dete{n-k+1,...,n}{n-k+1,...,n}
(A \cdot B^{-1})\right].
\]
\end{lemma}

\begin{proof} Straightforward using block matrices. 

\end{proof}
This lemma guarantees that the concept of preserving transverse orientation 
diffeomorfism
is well defined in terms of  subdeterminants since it is independent of the 
local 
coordinate sistem in a transversely oriented atlas. 

In general, given a $k\times n$ real matrix $C$, with $k,n \in \N$, we denote by
$C_{j_1...j_{r}}$ the $k\times r$ submatrix obtained by selecting the $r\leq n$ 
columns
${j_1...j_{r}}$ from the original matrix $C$. Analogously,   $C^{i_1...i_{r}}$,
with $r\leq k$, represents the $r\times n$ submatrix obtained by selecting the 
$r$ rows
${i_1, \ldots , i_{r}}$; yet the square matrix $C_{j_1...j_{r}}^{i_1...i_{r}}$, 
with $r
\leq \min \{k,n\}$ is obtained by selecting the indicated rows
and columns. The selection of rows commutes with the selection of columns.
With   
this notation $ \det(C)^{j_1...j_{r}}_{i_1...i_{r}}  = \det 
C^{j_1...j_{r}}_{i_1...i_{r}} $.  We use Cauchy-Binet formula  to find 
an alternative expression of the 
It\^o formula for the subdeterminants:

\begin{lemma}[Cauchy-Binet]\label{CauchyBinet} Let  $A$ be an $(l\times
m)$-matrix and  $B$ an $(m\times l)$-matrix, with  $l\leq m$. Then,
  $$
  \det (A\, B)=\sum_{i_1<i_2<...<i_{l }}
 \det({A}_{i_1...i_{l}})\cdot \det ({B}^{i_1...i_{l}}).
 $$
\end{lemma}
For a proof, see e.g., among many others,  Tracy and Widom \cite{Tracy}.

\subsection{It\^o-Liouville Formula for subdeterminants}

Consider initially that the Stratonovich SDE (\ref{equacao original}) 
takes place in an Euclidean space $\R^n$ with the associated stochastic flow of 
local diffeomorphisms $\varphi_t$. For  equations in  a Riemaninan 
manifold, using local coordinates,  the results  hold, up 
to the corresponging exit stopping times.
Consider the linearized flow $Y_t=D\varphi_t:\R^n
\rightarrow \R^n$, which  satisfies:
\begin{equation}
\label{eqY}dY_t=[D X_0(x_t)]Y_t\ dt
+ \sum_{r=1}^{m}[D X_r (x_t)]Y_t\circ dW^{r}_t.
\end{equation}
where $D X_{r}$ are the derivative matrices of the
corresponding vector fields relative to the canonical basis (covariant 
derivative $\nabla$, respectively, in a Riemannian manifold).

We use the following notation: given two $n\times n$ matrices $A$ 
and
$B$,  denote by  $[A:B :j]$, with $1\leq j\leq n$,  the matrix  obtained 
by
replacing the $j$-th row of $A$ by the
 the $j$-th row of $B$. Hence, for example
\begin{equation}\label{prodOp}
[I_n: DX_{i}:j] =\left(
\begin{array}{ccc}I_{j-1}& &0\\
 && \\ \displaystyle \frac{\partial }{\partial x_{1 }}X_{i}^{j} & \ldots 
& \displaystyle \frac{\partial }{\partial x_{n }}X^{j}_{i}\\
 && \\
0& &I_{n-j-1}
\end{array}\right),
\end{equation}
where `0' above represents zero submatrices with appropriate dimensions.
We also recall  the Laplace formula for the derivatives of the subdeterminants:
\begin{equation} \label{Eq: Laplace}
 \frac{\partial }{\partial x_{i_p j_q}} \dete{i_1, \ldots, i_k}{j_1, 
\ldots,j_k} (C) =
(-1)^{p+q} \dete{i_1,...,\widehat{i_p}, \ldots, i_k}{j_1,...,\widehat{j_q}, 
\ldots
,j_k} (C),
\end{equation}
with $1 \leq p,q \leq k$, 
where the
symbols  $\widehat{i_p}, \widehat{j_q}$ mean that these terms have been
excluded.

\begin{theorem}\label{Thm: Ito formula} For fixed $k$ rows $(i_1, \ldots, i_k)$ 
and $k$ columns $(j_1, 
\ldots,
j_k)$, in the domain of existence of the 
stochastic flow of local diffeomorphism, we have the following It\^o formula for
the corresponding subdeterminant of the linearized flow:
\begin{equation} \label{Eq: Ito-Liouville}
\begin{array}{ccl}
\dete{i_1...i_k}{j_1...j_k}(Y_t)&=& \dete{i_1...i_k}{j_1...j_k}(I_n) \  + \ \displaystyle \sum_{p=1}^k
\int_0^t \dete{i_1...i_k}{j_1...j_k}
[Y_s:(DX_{0})\cdot Y_s: i_p ] \ ds \\
&&  \\ 
& & +\displaystyle \sum_{r=1}^{m}\  \displaystyle \sum_{p=1}^k
\int_0^t \dete{i_1...i_k}{j_1...j_k}
[Y_s:(DX_{r})\cdot Y_s: i_p ]\circ dW^{r}_s.
\end{array}
\end{equation}

\end{theorem}
\begin{proof} For sake of notation,  use  
$G(C):=\det^{i_1...i_k}_{j_1...j_k}(C)$
 where $C$ is an $n\times
 n$
 real valued matrix. Applying It\^o formula, we have that:
\begin{equation}\label{itodet}
G(Y_t)=G(I)+ \displaystyle\sum_{p, q\leq k}\int_0^t \frac{\partial }{\partial
x_{i_pj_q}}G(Y_s)
\circ dY_{i_pj_q}(s).
\end{equation}
In coordinates, for all $1 \leq i,j \leq n$, equation (\ref{eqY}) leads to:
\begin{equation} \label{eqcoordY}
dY_{i j }(t)
= \displaystyle\sum_{l=1}^n \displaystyle \frac{\partial
}{\partial x_{l}}X^{i}_{0}(x_t) Y_{l j}(t)\ dt
+\sum_{r=1}^{m}\ \sum_{l=1}^n\frac{\partial }{\partial x_{l 
}}X_{r}^{i}(x_t)Y_{l
j}(t)\circ dW^{r}_t.
\end{equation}
From Equations (\ref{itodet}) and (\ref{eqcoordY})  we have that

\[
\begin{array}{ll}G(Y_t)=&G(I) +  \displaystyle \sum_{l=1}^n 
\sum_{p, q\leq k} \int_0^t \frac{\partial }{\partial
x_{i_pj_q}}G(Y_s)
\  \frac{\partial
}{\partial x_{l}}X^{i_p}_{0}(x_s) Y_{l j_p}(s)\ ds
 \\
 & \\
&+ \displaystyle\sum_{r=1}^{m}\ \sum_{l=1}^n \sum_{p, q\leq k} \int_0^t 
\frac{\partial }{\partial
x_{i_pj_q}}G(Y_s)
\  \frac{\partial
}{\partial x_{l}}X^{i_p}_{r}(x_s) Y_{l j_p}(s) \circ dW^r _s.
\end{array}
\] 
By Laplace formula (\ref{Eq: Laplace}), we have that
\[  
G([Y_s:(DX_{r})\cdot Y_s: i_p ])
=\sum_{q\leq k}\frac{\partial }{\partial
x_{i_pj_q}}G(Y_t) \ \sum_{l=1}^n \frac{\partial X_{r}^{i_p}}{\partial x_{l
}}(x_t)Y_{l j_q}.
\]
Hence formula (\ref{Eq: Ito-Liouville}) of the statement follows.
 
\end{proof}

In particular, for the maximal subdeterminant, i.e. with $k=n$,  one recovers the Liouville formula for 
determinants since for all  $r$ in  $\{0,1, \ldots, m\}$ we have that 
\[
  \displaystyle \sum_{p=1}^n
 \dete{i_1,...,i_n}{j_1,...,j_n}
[Y_s:(DX_{r})\cdot Y_s: i_p ] = \det (Y_s) \cdot 
\mathrm{Tr}(DX_r) \ \mathrm{sgn}(i) \ \mathrm{sgn}(j),
\]
where $\mathrm{sgn}(i)$ and $\mathrm{sgn}(j)$  are the parities of the permutations $(i_1,...,i_n)$ and $(j_1,...,j_n)$, respectively.
\begin{corollary}\label{Cor: Ito formula + Cauchy-Binet} For fixed $k$ rows 
$(i_1, \ldots, i_k)$ 
and $k$ columns $(j_1, 
\ldots,
j_k)$, in the domain of existence of the 
stochastic flow of local diffeomorphism, we have the following It\^o formula for
the corresponding subdeterminant of the linearized flow:

\begin{equation} \label{Eq: Ito-Liouville Cauchy-Binet}
\begin{array}{ccl}
\dete{i_1...i_k}{j_1...j_k}(Y_t)&=& \dete{i_1...i_k}{j_1...j_k}(I_n) \  + \displaystyle \sum_{l_1<...<l_k}\  \displaystyle \sum_{p=1}^k
\int_0^t \dete{i_1...i_k}{l_1..l_k}
[I: DX_{0}: i_p ] \cdot \dete{l_1...l_k}{j_1...j_k}(Y_s) \ ds \\
&&  \\ 
& & +\displaystyle \sum_{r=1}^{m}\  \displaystyle \sum_{l_1<...<l_k} \ \displaystyle \sum_{p=1}^k
\int_0^t \dete{i_1...i_k}{l_1...l_k}
[I: DX_{r}: i_p ] \cdot \dete{l_1...l_k}{j_1...j_k}(Y_s) \circ dW^{r}_s.
\end{array}
\end{equation}
\end{corollary}
\begin{proof} For each $i_p$, $1\leq p \leq k$ we have that
\[
[Y_s: DX_{r} \cdot Y_s: i_p]=[I: D X_{r}: i_p]\cdot 
Y_s.
\]
Then, apply Cauchy-Binet formula, Lemma (\ref{CauchyBinet}), to each product $[I:DX_{i}: 
j]\cdot
Y_s$.

\end{proof}

In this article we are particularly interested on applying formulae (\ref{Eq: Ito-Liouville}) and  (\ref{Eq: Ito-Liouville Cauchy-Binet})  for the
subdeterminant of the  lower--right $k\times k$ submatrices of $Y_t$,  cf. 
Proposition
\ref{Prop: analyt local obstruct decomp} and Theorem \ref{Thm: globdec}.

\subsection{Applications on decomposition of flows}

Suppose that a Riemannian manifold $M$ is endowed with a horizontal foliation which is transversely orientable and assume that the stochastic flow $\varphi_t$ is complete.  Theorem \ref{Thm: globdec} says that the global decomposition of the  stochastic
flow $\varphi_t$ is determined  by the subdeterminant
$\det_{n-k+1,...,n}^{n-k+1,...,n}(D\varphi_t)$, with respect to a biregular positively tranversal orientable coordinate system. Corollary \ref{Cor: Ito formula + Cauchy-Binet} states that this subdeterminant can be written in terms of  products of simpler subdeterminants. In our application here, if some of these subdeterminants vanish, then the equation for the relevant subdeterminant is a linear SDE, cf. Proposition \ref{Prop: sufficient condition} below.  The geometrical meanings are given after the proposition. 
We have the following sufficient condition on the subdeterminants  for the existence of global decomposition  for all $t\geq 0$:

\begin{proposition}  \label{Prop: sufficient condition} Assume that the 
 $(M,
\mathcal{H}, \mathcal{V})$ is transversely
orientable for the horizontal
foliation. Suppose that for all $0\leq r \leq m$  we have that 
 \begin{equation}\label{eq: hipo}
\dete{n-k+1,..., n}{l_1,...,l_k}
[I: DX_{r}: i ] \cdot \dete{l_1,...,l_k}{n-k+1, ..., n}(Y_t) = 0,
\end{equation}
for all $(n-k+1)\leq i\leq n $, all $(l_1>...>l_k)$ except possibly $(l_1,...,l_k)= (n-k+1, \ldots, n)$ and all $t\geq 0$. Then $\varphi_t$ is
decomposable for all $t\geq 0.$
\end{proposition}

\begin{proof}The hypotesis (\ref{eq: hipo}) together
with the fact that $\varphi_0=Id$ implies
that equation  (\ref{Eq: Ito-Liouville Cauchy-Binet}) reduces to
$$\begin{array}{ll}
\dete{n-k+1,...,n}{n-k+1,...,n}(Y_t)=&1+
\displaystyle\sum_{i= n-k+1}^{n}
\int_0^t\dete{n-k+1, \ldots ,n}{n-k+1,...,n}
([I:\nabla X_0:i])\dete{n-k+1,...,n}{n-k+1,...,n}
(Y_s)\ ds
\\
& +\displaystyle \sum_{r=1}^{m}\ \sum_{i= n-k+1}^{n}
\int_0^t\dete{n-k+1, \ldots ,n}{n-k+1,...,n}
([I:\nabla X_{r}:i])\dete{n-k+1,...,n}{n-k+1,...,n}
(Y_s)\circ dW^r_s 
\end{array}$$
Therefore $\det _{n-k+1, ..., n}^{n-k+1, \ldots , n}(Y_t)$ is the solution of
a linear SDE, with nonzero initial condition, 
hence,
it never vanishes. The result follows by Theorem (\ref{Thm: globdec}).

\end{proof}

\begin{corollary} If the stochastic flow  $\varphi_t$ satisfies 
$\varphi_t(\mathcal{V}(p))\subset \mathcal{V}(\varphi_t(p))$
for all $p\in M$ then $\varphi_t$ is decomposable globally for all $t\geq 0$.

\end{corollary}
\begin{proof} In fact, in this case, $\frac{\partial \varphi_t^i}{\partial x_j}=0$ for all $n-k+1 \leq j \leq n $ and $1\leq i\leq n-k$, hence 
\[
\dete{l_1,...,l_k}{n-k+1, ..., n}(Y_t) = 0,
\]
for all $(l_1>...>l_k)$ except $(l_1,...,l_k)= (n-k+1, \ldots, n)$. The result follows by the proposition above.

\end{proof}
The corollary above is illustrated by linear systems in $\R^n\setminus \{0\}$ with the spherical and radial foliations as the horizontal and vertical coordinates, respectively (cf. comments by the end of Prop. \ref{Prop: analyt local obstruct decomp}). 

\begin{corollary} If the covariant derivative $\nabla_v X_r$ is horizontal for any direction $v\in TM$,  
for all $0\leq r \leq  n$ then $\varphi_t$ is decomposable globally for all $t\geq 0$.

\end{corollary}
\begin{proof} In fact, in this case, 
\[
\dete{n-k+1,..., n}{l_1,...,l_k}
[I: DX_{r}: i ]=0, 
\]
for all $(n-k+1)\leq i\leq n $, for all $(l_1>...>l_k)$ including $(l_1,...,l_k)= (n-k+1, \ldots, n)$ and all $t\geq 0$. Then 
\[
\dete{n-k+1,...,n}{n-k+1,...,n}(Y_t) \equiv 1,
\]
The result follows by  Proposition \ref{Prop: sufficient condition}.

\end{proof}

\end{document}